\documentclass[12pt]{amsart}
\usepackage{a4}
\usepackage{booktabs}
\usepackage[T1]{fontenc}
\usepackage[all]{xy}
\usepackage{lipsum}
\usepackage{url}
\usepackage{tikz-network}
\usepackage{stackrel}
\usepackage{color}

\usepackage{amsmath,amsthm,amssymb}
\newcommand{\aspas}[1]{``{#1}''}
\usepackage{xcolor}
\usepackage{graphicx}
\usepackage{tikz}
\usetikzlibrary{
  hobby,
  intersections,
  spath3,
  decorations.markings,
  arrows.meta,
}
\newtheorem{theorem}{Theorem}[section]
\newtheorem{lemma}[theorem]{Lemma}
\newtheorem{example}[theorem]{Example}
\newtheorem{proposition}[theorem]{Proposition}
\theoremstyle{definition}
\newtheorem{definition}[theorem]{Definition}

\newtheorem{remark}[theorem]{Remark}
\newtheorem{corollary}[theorem]{Corollary}

\numberwithin{equation}{section}

\begin{document}

%%%%%%%%%%%%%%%%%%%%%%%%%%%%%%%%%%%%%%%%%%%%%%%%%%%%%%%%%%%%

\renewcommand{\bf}{\bfseries}
\renewcommand{\sc}{\scshape}
%insert defs/styles

\title[(Injective) facet-complexity between simplicial complexes]%
{(Injective) facet-complexity between simplicial complexes}

%    Information for first author:
\author[C. A. Ipanaque Zapata]{Cesar A. Ipanaque Zapata}
\address[C. A. Ipanaque Zapata]{Departamento de Matem\'atica, IME-Universidade de S\~ao Paulo, Rua do Mat\~ao, 1010 CEP: 05508-090, S\~ao Paulo, SP, Brazil}
\email{cesarzapata@usp.br}

%\thanks {The first author would like to thank grant \#2022/16695-7 and \#2023/16525-7 from the S\~{a}o Paulo Research Foundation (\textsc{fapesp}) for financial support.}

%    Information for second author (if needed):
\author[A. Borat]{Ayse Borat}%
\address[A. Borat]{Bursa Technical University, Faculty of Engineering and Natural Sciences, Department of Mathematics, Bursa, Turkiye}
\email{ayse.borat@btu.edu.tr}

%    General info
%%%%%%%%%%%%%%%%%%%%%%%%%%%%%%%%%%%%%%%%%%%%%%%%%%%
% 05E45  	Combinatorial aspects of simplicial complexes
% 05C15  	Coloring of graphs  and hypergraphs
%05C20  	Directed graphs  (digraphs), tournaments
% 05C51  	Graph designs and isomorphic decomposition
%  	05C60  	Isomorphism problems in graph theory (reconstruction conjecture, etc.) and homomorphisms (subgraph embedding, etc.)
%   05C90  	Applications of graph theory
%         Please use the current 2010 Mathematics Subject Classification:             %
%         http://www.ams.org/mathscinet/msc/                                                        %
%         http://www.zentralblatt-math.org/msc/en/ 
%%%%%%%%%%%%%%%%%%%%%%%%%%%%%%%%%%%%%%%%%%%%%%%%%%%


\begin{abstract} We present the notion of facet-complexity, $\text{C}(\mathsf{L};\mathsf{K})$, for two simplicial complexes $\mathsf{L}$ and $\mathsf{K}$, along with basic results for this numerical invariant. This invariant $\text{C}(\mathsf{L};\mathsf{K})$ quantifies the \aspas{complexity} of the following question: When does there exist a facet simplicial map $\mathsf{L}\to \mathsf{K}$? A facet simplicial map is a simplicial map that preserves non-unitary facets. Likewise, we introduce the notion of injective facet-complexity, $\text{IC}(\mathsf{L};\mathsf{K})$. These invariants generalize the notion of (injective) hom-complexity between graphs, recently introduced by Zapata et al.  We demonstrate a triangular inequality for (injective) facet-complexity and show that it is a simplicial complex invariant. Additionally, these invariants provide an obstruction to the existence of facet simplicial maps. We explore the sub-additivity of (injective) facet-complexity and we present a lower bound in terms of the chromatic number. Moreover, we provide an upper bound for $\mathrm{C}(\mathsf{L};\mathsf{H})$ in terms of the number of facets of $L$. Finally, we establish a  formula for $\mathrm{IC}(\mathsf{L};\mathsf{K})$ when $\mathsf{L}$ is a pure simplicial complex and $K$ is a complete simplicial complex.
\end{abstract}
\maketitle

%%%%%%%%%%%%%%%%%%%%%%%%%%%%%%%%%%%%%%%%%%%%%%%%%%%%%%%%%%%%%%

\section{Introduction}\label{secintro}% 
In this article, the term \aspas{simplicial complex} refers to an abstract simplicial complex. For more details, see Section~\ref{sec:pre}. The symbol $\lceil m\rceil$ denotes the least integer greater than or equal to $m$, while $\lfloor m\rfloor$ denotes the greatest integer less than or equal to $m$.    

\medskip Let $\mathsf{L}$ and $\mathsf{K}$ be simplicial complexes. For the purpose of this work, we present the following notion. A \textit{facet simplicial map} of $\mathsf{L}$ to $\mathsf{K}$, written as $f:\mathsf{L}\stackrel{\mathrm{facet}}{\to}\mathsf{K}$, is a simplicial map $f:\mathsf{L}\to \mathsf{K}$ such that $f(F)$ is a non-unitary facet (i.e., non-unitary maximal face) of $\mathsf{K}$ whenever $F$ is a non-unitary facet of $\mathsf{L}$ (see Definition~\ref{defn:facet-map}). The symbol $\mathsf{L}\stackrel{\mathrm{facet}}{\to} \mathsf{K}$ means that there is a facet simplicial map from $\mathsf{L}$ to $\mathsf{K}$; otherwise, we write $\mathsf{L}\stackrel{\mathrm{facet}}{\not\to} \mathsf{K}$, as explained in Section~\ref{sec:pre}. 
 
\medskip Given two simplicial complexes $\mathsf{L}$ and $\mathsf{K}$ without isolated vertices, it is natural to pose the following question: When is there a facet simplicial complex $\mathsf{L}\stackrel{\mathrm{facet}}{\to}\mathsf{K}$? 
In the case that both $\mathsf{L}$ and $\mathsf{K}$ are simplicial complexes of dimension one (i.e., they can be seen as graphs), this question represents a significant challenge in graph theory (see \cite{hell1990}).

\medskip Motivated by this question in graph theory, the notion of (injective) hom-complexity between graphs was recently introduced in \cite{zapata2024}. In this work, we extend this notion to higher dimensions. We introduce the notion of facet-complexity between two simplicial complexes  $\mathsf{L}$ and $\mathsf{K}$, denoted by $\text{C}(\mathsf{L};\mathsf{K})$ (Definition~\ref{defn:complexity}), along with its basic results. More precisely, $\text{C}(\mathsf{L};\mathsf{K})$ is defined as the least positive integer $\ell$ such that there are $\ell$ distinct subcomplexes $\mathsf{L}_j$ of $\mathsf{L}$ with $\mathsf{L}=\mathsf{L}_1\cup\cdots\cup \mathsf{L}_\ell$, and over each $\mathsf{L}_j$, there exists a facet simplicial map $\mathsf{L}_j\to \mathsf{K}$. For instance, we have $\text{C}(\mathsf{L};\mathsf{K})=1$ if and only if there is a facet simplicial map $\mathsf{L}\stackrel{\mathrm{facet}}{\to}\mathsf{K}$. Likewise, we introduce the notion of injective facet-complexity, $\text{IC}(\mathsf{L};\mathsf{K})$. 

\medskip Also, we discuss the notion of strict simplicial map, strict chromatic number (Remark~\ref{rem:strict-chroma}), and (injective) strict-complexity (Remark~\ref{rem:ineq-complex-q}). We believe that this strict version for facet-complexity coincides with the hom-complexity of underlying graphs.  

\medskip This work is also motivated by a fundamental curiosity to present a well-defined methodology that can help address the \aspas{complexity} of the data migration problem in higher cases \cite{hussein2021}, \cite{spivak2012}. 

\medskip The main results of this work are:
\begin{itemize}
    \item Introduction of the concepts of facet-complexity $\text{C}(\mathsf{L};\mathsf{K})$ and injective facet-complexity $\text{IC}(\mathsf{L};\mathsf{K})$ for two simplicial complexes $\mathsf{L}$ and $\mathsf{K}$ (Definition~\ref{defn:complexity}).
     \item A triangular inequality (Theorem~\ref{thm:inequality-three-graphs}). 
    \item The existence of a facet simplicial map implies inequalities between the facet-complexities. Likewise, the existence of an injective facet  simplicial map implies inequalities between the injective facet-complexities (Theorem~\ref{prop:complexity-subgraphs}). In particular, this shows  that (injective) facet-complexity is a simplicial complex invariant. It also implies that (injective) facet-complexity provides a numerical obstruction to the existence of a facet simplicial map. 
    \item Sub-additivity (Theorem~\ref{thm:category-union}).
    \item A lower bound (Theorem~\ref{thm:lower-bound}).
    \item An upper bound (Theorem~\ref{thm:uper-bound}).
    \item A formula for $\mathrm{IC}(\mathsf{L};\mathsf{K})$ whenever $\mathsf{L}$ is a pure simplicial complex and $K$ is a complete simplicial complex (Proposition~\ref{prop:k-complete-facet}). 
\end{itemize} 

The paper is organized as follows: We begin with a brief review of simplicial complexes and facet simplicial maps (Section~\ref{sec:pre}). We state and prove Proposition~\ref{prop:chromatic-union}, which is fundamental in Theorem~\ref{thm:lower-bound}. In Section~\ref{sec:one}, we introduce the notions of facet-complexity $\text{C}(\mathsf{L};\mathsf{K})$ and injective facet-complexity $\text{IC}(\mathsf{L};\mathsf{K})$ for two simplicial complexes $\mathsf{L}$ and $\mathsf{K}$ (Definition~\ref{defn:complexity}). Theorem~\ref{prop:complexity-subgraphs} presents inequalities between the facet-complexities under the existence of a facet simplicial map. In particular, Corollary~\ref{cor:invariant-iso-complexity} shows that facet-complexity is a simplicial complex invariant. Proposition~\ref{prop:no-existencia} states that facet-complexity provides a numerical obstruction to the existence of a facet simplicial map. Furthermore, Theorem~\ref{thm:category-union} demonstrates the sub-additivity of (injective) facet-complexity. A lower bound in terms of chromatic number is provided in Theorem~\ref{thm:lower-bound}. Additionally, Theorem~\ref{thm:uper-bound}  provides an upper bound for $\mathrm{C}(\mathsf{L};\mathsf{H})$ in terms of the number of facets of $L$. Also, in Proposition~\ref{prop:k-complete-facet}, we present a formula for $\mathrm{IC}(\mathsf{L};\mathsf{K})$ whenever $K$ is a complete simplicial complex. We close this section with Remark~\ref{rem:future-work}, which presents directions for future work.

%%%%%%%%%%%%%%%%%%%%%%%%%%%%%%%%%%%%%%%%%%%%%%%%%%%%%%%%%%%%%%%%%%%%%%%%%%%%%%%%%%%%%%%%%%%%%%%%%%%%%%%%%%%%%%%%%%%%%%%%%%%%
\section{Simplicial complexes and facet simplicial maps} \label{sec:pre}
In this section, we recall some definitions and we fix the notations. We follow the standard notation for simplicial complexes as used in \cite[Section 1.5, p. 13]{matousek2003}. An \textit{abstract simplicial complex} is a pair $(V,\mathsf{K})$, where $V=V(\mathsf{K})$ is a set of \textit{vertices}, and $\mathsf{K}\subseteq 2^V$ is a set of \textit{simplices}, such that if $F\in \mathsf{K}$ and $G\subseteq F$, then $G\in \mathsf{K}$ \cite[Definition 1.5.1, p. 13]{matousek2003}. In this case, such $G$ is called a \textit{face} of the simplex $F$. A \textit{facet} is a maximal simplex, i.e., any simplex in a complex that is not a face of any larger simplex. In this article, the term \aspas{simplicial complex} refers to an abstract simplicial complex.  

\medskip Usually we may assume that $V=\bigcup \mathsf{K}$; thus it suffices to write $\mathsf{K}$ instead of $(V,\mathsf{K})$, where $V$ is understood to equal $\bigcup \mathsf{K}$. 

\medskip Let $\mathsf{K}$ be a simplicial complex. The \textit{dimension} of a simplex $F\in \mathsf{K}$ is given by $\dim(F)=|F|-1$, and the \textit{dimension} of $\mathsf{K}$ by $\dim(\mathsf{K})=\max\{\dim(F):\quad F\in \mathsf{K}\}$ \cite[Definition 1.5.1, p. 13]{matousek2003}.  

\medskip We shall use the simplified notation for simplices, where $v_1\cdots v_m$ represents the simplex $\{v_1,\ldots,v_m\}$. We have $v_1\cdots v_m=v_{\sigma(1)}\cdots v_{\sigma(m)}$ for any permutation $\sigma\in S_m$. If $u,v\in F$ for some simplex $F$, we say that $u$ and $v$ are \textit{adjacent}, we also say that $u$ and $v$ are \textit{neighbours}. If $u,v\in F$ for some $d$-dimensional simplex $F$, we say that $u$ and $v$ are $d$-adjacent. The number of neighbours of $v$ (other than $v$) is called the \textit{degree} of $v$; the number of $d$-neighbours of $v$ (other than $v$) is called the \textit{$d$-degree} of $v$. Furthermore, $\text{deg}_d(v)$, and $\text{deg}(v)$ denote the $d$-degree, and degree of vertex $v$, respectively. Note that $\text{deg}(v)\leq\sum_{d\geq 1}\text{deg}_d(v)$. A vertex $v$ is called \textit{isolated} if $\text{deg}(v)=0$.  

\medskip We say that a simplicial complex $\mathsf{L}$ is a \textit{subcomplex} of $\mathsf{K}$ if $\mathsf{L}\subseteq \mathsf{K}$ (and of course $V(\mathsf{L})\subseteq V(\mathsf{K})$). A subcomplex $\mathsf{L}$ of $\mathsf{K}$ is called a \textit{spanning subcomplex} if $V(\mathsf{L})=V(\mathsf{K})$. Additionally, $\mathsf{L}$ is an \textit{induced subcomplex} of $\mathsf{K}$ if it is a subcomplex of $\mathsf{K}$ and contains all the simplices of $\mathsf{K}$ among the vertices in $\mathsf{L}$. We say that a simplicial complex $\mathsf{K}$ is \textit{complete} if $\mathsf{K}=2^V$. A \textit{clique} in a simplicial complex $\mathsf{K}$ is a complete subcomplex of $\mathsf{K}$. The symbol $\Gamma_n$ denotes the complete simplicial complex on $n$ vertices, while $\mathsf{K}_n$ denotes the simplicial complex on $n$ vertices given by $\mathsf{K}_n=\{F\in 2^n:\quad |F|\leq n-1\}$. Note that $\dim(\Gamma_n)=n-1$ and $\dim(\mathsf{K}_n)=n-2$. Furthermore, $V(\Gamma_n)$ is the only facet of $\Gamma_n$, while $\mathsf{K}_n$ has $n$ facets.  

\medskip  Let $\mathsf{L}$ and $\mathsf{K}$ be simplicial complexes. A \textit{simplicial map} of $\mathsf{L}$ to $\mathsf{K}$, written as $f:\mathsf{L}\to \mathsf{K}$, is a mapping $f:V(\mathsf{L})\to V(\mathsf{K})$ such that $f(F)\in \mathsf{K}$  whenever $F\in \mathsf{L}$ \cite[Definition 1.5.2, p. 14]{matousek2003}. We call a simplicial map $f:\mathsf{L}\to \mathsf{K}$ \textit{injective}, \textit{surjective}, or \textit{bijective} if the mapping $f:V(\mathsf{L})\to V(\mathsf{K})$ is injective, surjective, or bijective, respectively. A bijective simplicial map $f:\mathsf{L}\to \mathsf{K}$ whose inverse map $f^{-1}:V(\mathsf{K})\to V(\mathsf{L})$ is also a simplicial map is called an \textit{isomorphism}, and that $\mathsf{L}$ and $\mathsf{K}$ are \textit{isomorphic}. 

\medskip For this paper, we introduce the following concept. 

\begin{definition}[(Facet) strict simplicial map]\label{defn:facet-map}
  Let $\mathsf{L}$ and $\mathsf{K}$ be simplicial complexes. \begin{enumerate}
  \item[(1)] A \textit{facet simplicial map} of $\mathsf{L}$ to $\mathsf{K}$, written as $f:\mathsf{L}\stackrel{\mathrm{facet}}{\to} \mathsf{K}$, is a simplicial map $f:\mathsf{L}\to\mathsf{K}$ such that $f(F)$ is a non-unitary facet of $\mathsf{K}$ whenever $F$ is a non-unitary facet of $\mathsf{L}$. Hence, facet simplicial maps of simplicial complexes is a simplicial map preserving the non-unitary facets. 
      \item[(2)] A \textit{strict simplicial map} of $\mathsf{L}$ to $\mathsf{K}$, written as $f:\mathsf{L}\stackrel{\text{s}}{\to} \mathsf{K}$, is a mapping $f:V(\mathsf{L})\to V(\mathsf{K})$ such that $f(F)\in \mathsf{K}$ is a simplex of dimension $d$ whenever $F\in \mathsf{L}$ is a simplex of dimension $d$. Hence, strict simplicial maps of simplicial complexes preserve the dimension of the simplices. Note that, any strict simplicial map is a simplicial map. Also, any injective simplicial map is a strict simplicial map.    
  \end{enumerate} 
\end{definition} 

The symbol $\mathsf{L}\stackrel{\mathrm{facet}}{\to} \mathsf{K}$ indicates that there exists a facet simplicial map from $\mathsf{L}$ to $\mathsf{K}$, and in this case, we say that $\mathsf{L}$ is \textit{facet $\mathsf{K}$-colourable}; otherwise, we write $\mathsf{L}\stackrel{\text{facet}}{\not\to} \mathsf{K}$. Likewise, $\mathsf{L}\stackrel{\text{s}}{\to} \mathsf{K}$ indicates that there exists a strict simplicial map from $\mathsf{L}$ to $\mathsf{K}$, and in this case, we say that $\mathsf{L}$ is \textit{strict $\mathsf{K}$-colourable}; otherwise, we write $\mathsf{L}\stackrel{\text{s}}{\not\to} \mathsf{K}$. Observe that if $\mathsf{L}\stackrel{\text{s}}{\to} \mathsf{K}$, then $\dim(\mathsf{L})\leq \dim(\mathsf{K})$. Also, note that if we remove or add isolated vertices from $\mathsf{L}$, its (facet, respectively) strict $\mathsf{K}$-colorability does not change. Given a subcomplex $\mathsf{L}$ of $\mathsf{K}$, the inclusion map $V(\mathsf{L})\hookrightarrow V(\mathsf{K})$ is a strict simplicial map and is called the \textit{inclusion simplicial map} $\mathsf{L}\hookrightarrow \mathsf{K}$. 

\medskip We call a facet simplicial map $f:\mathsf{L}\stackrel{\mathrm{facet}}{\to} \mathsf{K}$ \textit{injective}, \textit{surjective}, or \textit{bijective} if the simplicial map $f:\mathsf{L}\to \mathsf{K}$ is injective, surjective, or bijective, respectively. A bijective facet simplicial map $f:\mathsf{L}\stackrel{\mathrm{facet}}{\to} \mathsf{K}$ whose inverse map $f^{-1}:V(\mathsf{K})\to V(\mathsf{L})$ is also a facet simplicial map is called a \textit{facet isomorphism}, and that $\mathsf{L}$ and $\mathsf{K}$ are \textit{facet isomorphic}. Note that if $f$ is an isomorphism, then $f$ and its inverse $f^{-1}$ are facet simplicial maps. Hence, facet isomorphisms coincide with isomorphisms. 

\medskip Let $\mathsf{K}$ be a simplicial complex. For each $d\geq 0$, let \[F_d(\mathsf{K})=\{F\in \mathsf{K}:\quad \dim(F)=d\},\] the set of all $d$-dimensional simplices of $\mathsf{K}$. Note that $F_0(\mathsf{K})$ corresponds to the vertices $V(\mathsf{K})$ (note that $V(\mathsf{K})=\bigcup F_0(\mathsf{K})$) and $F_d(\mathsf{K})=\emptyset$ for any $d>\dim(\mathsf{K})$. Furthermore, $\mathsf{K}=\bigcup_{d\geq 0}F_d(\mathsf{K})$. 

\medskip A strict simplicial map $f:\mathsf{L}\stackrel{\text{s}}{\to} \mathsf{K}$ is a mapping from $V(\mathsf{L})$ to $V(\mathsf{K})$, but since it preserves the dimension of the simplices, it also naturally defines a mapping $f^{\#}:=f^{\#}_d:F_d(\mathsf{L})\to F_d(\mathsf{K})$ by setting $f^{\#}(F)=f(F)$ for all $F\in F_d(\mathsf{L})$. We call a strict simplicial map   $f:\mathsf{L}\stackrel{\text{s}}{\to} \mathsf{K}$ \textit{$d$-injective}, \textit{$d$-surjective}, or \textit{$d$-bijective} if the mapping $f^{\#}:F_d(\mathsf{L})\to F_d(\mathsf{K})$ is injective, surjective, or bijective, respectively. A $0$-injective, surjective, or bijective, we call a \text{vertex-injective}, \text{vertex-surjective}, or \text{vertex-bijective}, respectively. A strict simplicial map $f$ is an \textit{injective strict simplicial map}, a \textit{surjective strict simplicial map }, or a \textit{bijective strict simplicial map} if, for each $d$, it is $d$-injective, surjective, or bijective, respectively. Note that if $f:\mathsf{L}\stackrel{\text{s}}{\to} \mathsf{K}$ is a bijective strict simplicial map, the inverse map $f^{-1}:V(\mathsf{K})\to V(\mathsf{L})$ is a strict simplicial map from $\mathsf{K}$ to $\mathsf{L}$, and in this case, we say that $f:\mathsf{L}\stackrel{\text{s}}{\to} \mathsf{K}$ is a \textit{strict isomorphism}, and that $\mathsf{L}$ and $\mathsf{K}$ are \textit{strict isomorphic}. Note that strict isomorphisms coincide with isomorphisms.   

\medskip Note that a strict simplicial map that is vertex-injective is also $d$-injective for each $d\geq 1$ (but not conversely), and as long as $\mathsf{L}$ has no isolated vertices, a strict simplicial map that is $d$-surjective for each $d\geq 1$ is also vertex-surjective (but not conversely). In other words, injective strict simplicial maps are the same as vertex-injective strict simplicial maps, while surjective strict simplicial maps are, in the absence of isolated vertices, the same as $d$-surjective strict simplicial maps for each $d\geq 1$. 

\medskip The following statement is straightforward to verify.

\begin{lemma}\label{lem:hom-degree}
If $f:\mathsf{L}\to \mathsf{K}$ is an injective simplicial map (and of course it is an injective strict simplicial map) and $v\in V(\mathsf{L})$. Then: 
\begin{enumerate}
    \item[(1)] $\mathrm{deg}(f(v))\geq \mathrm{deg}(v)$.
     \item[(2)] $\mathrm{deg}_d(f(v))\geq \mathrm{deg}_d(v)$ for each $d\geq 1$.
\end{enumerate}
\end{lemma}

Now, we recall the definition of the union of simplicial complexes.

\begin{definition}[Union of Simplicial Complexes]\label{defn:union-graphs}
\noindent\begin{enumerate}
    \item[(1)] Let $\mathsf{L}_1, \mathsf{L}_2,\ldots, \mathsf{L}_n$ be simplicial complexes. The \textit{union} $\mathsf{L}_1\cup\cdots\cup \mathsf{L}_n$ is defined by $V(\mathsf{L}_1\cup\cdots\cup \mathsf{L}_n)=V(\mathsf{L}_1)\cup\cdots\cup V(\mathsf{L}_n)$, and $F_d(\mathsf{L}_1\cup\cdots\cup \mathsf{L}_n)=F_d(\mathsf{L}_1)\cup\cdots\cup F_d(\mathsf{L}_n)$ for each $d\geq 1$.    
    \item[(2)] Let $\mathsf{L}$ be a simplicial complex and $\mathsf{A},\mathsf{B}$ be subcomplexes of $\mathsf{L}$ such that $V(\mathsf{A})\cap V(\mathsf{B})=\emptyset$ (and thus $F_d(\mathsf{A})\cap F_d(\mathsf{B})=\emptyset$ for each $d$). In this case, the union $\mathsf{A}\cup \mathsf{B}$ is called the \textit{disjoint union} and is denoted by $\mathsf{A}\sqcup \mathsf{B}$. Furthermore, given (facet, or strict, respectively) simplicial maps $f:\mathsf{A}\to \mathsf{K}$ and $g:\mathsf{B}\to \mathsf{K}$, the map $f\sqcup g:V(\mathsf{A})\cup V(\mathsf{B})\to V(\mathsf{K})$, defined by \[(f\sqcup g)(v)=\begin{cases}
    f(v),&\hbox{ if $v\in V(\mathsf{A})$,}\\
    g(v),&\hbox{ if $v\in V(\mathsf{B})$}, 
\end{cases}\] is a (facet, or strict, respectively) simplicial map of $\mathsf{A}\sqcup \mathsf{B}$ to $\mathsf{K}$ (using the fact that $F_d(\mathsf{A})\cap F_d(\mathsf{B})=\emptyset$ for each $d$). 
\end{enumerate} 
\end{definition}
 
\begin{definition}[Underlying Graph]\label{defn:associated-graph}
 Given a simplicial complex $\mathsf{L}$, the \textit{underlying graph} of $\mathsf{L}$ is given by $\mathsf{L}^\ast$ where $V(\mathsf{L}^\ast)=V(\mathsf{L})$ and $E(\mathsf{L}^\ast)=F_1(\mathsf{L})$.    
\end{definition}

 Observe that $(\Gamma_n)^\ast=(\mathsf{K}_n)^\ast=K_n$ for any $n\geq 3$, it is the usual complete graph on $n$ vertices.

\begin{remark}\label{rem:recover-graph-hom}
 Note that if $\mathsf{L}$ and $\mathsf{K}$ are $1$-dimensional simplicial complexes, then a map $f:V(\mathsf{L})\to V(\mathsf{K})$ is a facet simplicial map if and only if it is a graph homomorphism from $\mathsf{L}^\ast$ to $\mathsf{K}^\ast$.   
\end{remark}

 A \textit{$k$-coloring} of a simplicial complex $\mathsf{L}$ is an assignment of $k$ colors to the vertices of $\mathsf{L}$, such that no non-unitary facet (i.e., maximal face) of $\mathsf{L}$ is monochromatic, i.e., no non-unitary facet has all its vertices colored in one color. Suppose that the integers  $1,2,\ldots,k$ are used as the \aspas{colors} in the $k$-colorings. Then, a $k$-coloring of $\mathsf{L}$ can be viewed as a surjective mapping $f:V(\mathsf{L})\to \{1,2,\ldots,k\}$; the requirement that no non-unitary facet of $\mathsf{L}$ is monochromatic means that $f(v_i)\neq f(v_j)$ for some $i\neq j$ whenever $v_1\cdots v_n$ is a non-unitary (i.e., $n\geq 2$) facet of $\mathsf{L}$. On the other hand, if there exists a mapping $f:V(\mathsf{L})\to \{1,2,\ldots,k\}$ such that $f(v_i)\neq f(v_j)$ for some $i\neq j$ whenever $v_1\cdots v_n$ is a non-unitary  facet of $\mathsf{L}$, then $\mathsf{L}$ admits a $n$-coloring with $n\leq k$. Note that $n=k$ whenever $f:V(\mathsf{L})\to \{1,2,\ldots,k\}$ is surjective. 

\medskip The \textit{chromatic number} of $\mathsf{L}$, denoted by $\chi(\mathsf{L})$, is defined as the smallest $k$ such that $\mathsf{L}$ admits a $k$-coloring. In other words, the chromatic number of $\mathsf{L}$ is the least number of colors needed to color the vertices of $\mathsf{L}$ in such a way that no non-unitary facet of $\mathsf{L}$ has all its vertices colored in one color \cite[p. 957]{golubev2017}. Note that if we remove or add isolated vertices from $\mathsf{L}$, its chromatic number does not change.

\begin{example}\label{exam:chromatic-numberex}
Let $\mathsf{L}$ and $\mathsf{K}$ be simplicial complexes defined as follows: $V(\mathsf{L})=\{a,b,c,d,e\}$, $F_1(\mathsf{L})=\{ab,bc,ac,cd,de,ce\}$, $F_2(\mathsf{L})=\{abc\}$, and  $V(\mathsf{K})=\{a',b',c',d'\}$, $F_1(\mathsf{K})=\{a'b',b'c',c'a',c'd'\}$, $F_2(\mathsf{K})=\{a'b'c'\}$. 
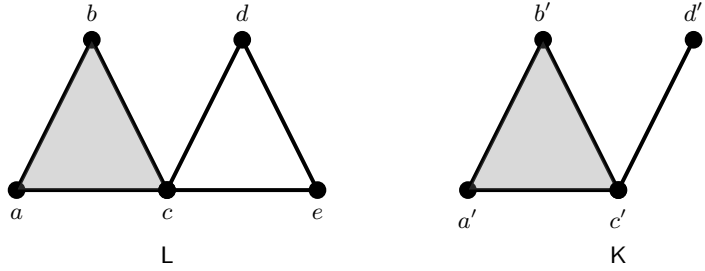
    $$
\begin{tikzpicture}
% Simple L
\Vertex[x=0,y=0,size=0.2,label=$a$,position=below,color=black]{A} 
\Vertex[x=1, y=2, size=0.2,label=$b$,position=above,color=black]{B}
\Vertex[x=2, y=0, size=0.2,label=$c$,position=below,color=black]{C} 
\Vertex[x=3, y=2, size=0.2,label=$d$,position=above,color=black]{D}
\Vertex[x=4, y=0, size=0.2,label=$e$,position=below,color=black]{E} 
% Filling the triangle A, B, C with a color
 \fill[black!50, draw = black, opacity=0.3] (0,0) -- (1,2) -- (2,0) -- cycle;
\Edge[color=black](A)(B) 
\Edge[color=black](B)(C)
\Edge[color=black](C)(A)
\Edge[color=black](C)(D)
\Edge[color=black](C)(E)
\Edge[color=black](D)(E)
\Vertex[x=2,y=0,size=0.2,distance=0.5cm,label=$\mathsf{L}$,position=below,color=black]{M} 
%SIMPLEX K
\Vertex[x=6,y=0,size=0.2,label=$a'$,position=below,color=black]{F} 
\Vertex[x=7,y=2,size=0.2,label=$b'$,position=above,color=black]{G} 
\Vertex[x=8,y=0,size=0.2,label=$c'$,position=below,color=black]{H} 
\Vertex[x=9,y=2,size=0.2,label=$d'$,position=above,color=black]{I} 
% Filling the triangle A', B', C' with a color
 \fill[black!50, draw = black, opacity=0.3] (6,0) -- (7,2) -- (8,0) -- cycle;
 \Edge[color=black](F)(G) 
\Edge[color=black](G)(H)
\Edge[color=black](H)(F)
\Edge[color=black](H)(I)
\Vertex[x=8, y=0, size=0.2,distance=0.5cm,label=$\mathsf{K}$,position=below,color=black]{N}
  \end{tikzpicture}
 $$ We have that $\chi(\mathsf{L})=3$ and $\chi(\mathsf{K})=2$. For example, we have the colorings $f:V(\mathsf{L})\to\{1,2,3\}$ and $g:V(\mathsf{K})\to\{1,2\}$ given by \begin{align*}
     f(a)&=f(d)=1,\\
     f(b)&=f(c)=2,\\
     f(e)&=3,
 \end{align*} and \begin{align*}
     g(a')&=g(d')=1,\\
     g(b')&=g(c')=2.
 \end{align*}  
\end{example}

\begin{remark}\label{rem:correct}
 We observe that the inequality $\chi(\mathsf{L})\leq\lceil \chi(\mathsf{L}^\ast)/\dim(\mathsf{L})\rceil$ presented in \cite[Proposition 2.1, p. 957]{golubev2017} is not correct for any simplicial complex $\mathsf{L}$, where $\chi(L^\ast)$ denotes the usual chromatic number of the graph $\mathsf{L}^\ast$. For instance, consider the simplicial complexes given in Example~\ref{exam:chromatic-numberex}.    
\end{remark} 

\medskip Note that, for any face $F$ of a simplicial complex $\mathsf{L}$, the inequality \begin{equation}\label{eqn:dim-face-chr-1-sk}
    \dim(F)+1\leq \chi(\mathsf{L}^\ast)
\end{equation} always hold. 

\medskip Given a simplicial complex $\mathsf{L}$, let $G:=G_{\mathsf{L}}$ be the underlying graph of the collection $\mathcal{S}=\{F:~ \text{ $F$ is a facet of $\dim(F)=1$}\}$, i.e., $V(G)=\bigcup_{F\in\mathcal{S}} F$ and $E(G)=\mathcal{S}$. We present the following inequalities.

\begin{proposition}\label{prop:chrom-graph-chro}
  For any simplicial complex $\mathsf{L}$, we have \[\chi(G_{\mathsf{L}})\leq \chi(\mathsf{L})\leq\lceil \chi(\mathsf{L}^\ast)/d\rceil, \] whenever $d:=\min\{\dim(F):~ \text{ $F$ is a facet of $\mathsf{L}$}\}>0$. 
\end{proposition}
\begin{proof}
    Let  $n=\chi(L^\ast)$ and $m=\lceil \chi(L^\ast)/d\rceil$ (i.e., $m-1<n/d\leq m$). Observe that $m>1$ because $n>d$ (see inequality~(\ref{eqn:dim-face-chr-1-sk})). Let $g:\mathsf{L}^\ast\to K_n$ be a surjective graph homomorphism. Recall that $V(\mathsf{L})=V(\mathsf{L}^\ast)$. Set $V(\mathsf{L})=A_1\sqcup\cdots\sqcup A_n$, where each $A_j=g^{-1}(j)$. Here $\sqcup$ means the usual disjoint union of sets. Since $(m-1)d+1\leq n$, we can take the following subsets $B_1=A_1\sqcup\cdots \sqcup A_d$, $B_2=A_{d+1}\sqcup\cdots\sqcup A_{2d}$,$\ldots$, $B_m=A_{(m-1)d+1}\sqcup\cdots\sqcup A_{n}$. Of course, $V(\mathsf{L})=B_1\sqcup\cdots\sqcup B_m$. Note that each $B_1,\ldots,B_{m-1}$ is the disjoint union of $d$ subsets $A_j$, and $B_m$ is the disjoint union of $\ell$ subsets $A_j$, where  $\ell=n-(m-1)d-1+1=n-md+d\leq d$. We consider the map $f:V(\mathsf{L})\to \{1,\ldots,m\}$ by \[f(v)=i \text{ whenever $v\in B_i$}.\] If $v_1\cdots v_k$ is a non-unitary (i.e., $k\geq 2$) facet of $\mathsf{L}$, we have $k\geq d+1$. Then, we can conclude that $f(v_r)\neq f(v_s)$ for some $r\neq s$. Otherwise, $v_1,\ldots, v_k\in B_i$ for some $i\in\{1,\ldots,m\}$, then there exists $p,q\in \{1,\ldots,k\}$ with $p\neq q$ such that $v_p,v_q\in A_j$ for some $j\in\{1,\ldots,n\}$, i.e., $g(v_p)=g(v_q)=j$, which is a contradiction, because $v_pv_q\in E(\mathsf{L}^\ast)$ and $g$ is a graph homomorphism.  Therefore, $\chi(\mathsf{L})\leq m=\lceil \chi(\mathsf{L}^\ast)/d\rceil$.

    Now, we will check the inequality $\chi(G_{\mathsf{L}})\leq \chi(\mathsf{L})$. Let $m=\chi(\mathsf{L})$, and consider a surjective mapping $f:V(\mathsf{L})\to \{1,\ldots,m\}$ such that $f(v_i)\neq f(v_j)$ for some $i\neq j$ whenever $v_1\cdots v_n$ is a non-unitary (i.e., $n\geq 2$) facet of $\mathsf{L}$. Since $V(G_{\mathsf{L}})\subseteq V(\mathsf{L})$, we consider the restriction map $f_|:V(G_{\mathsf{L}})\to \{1,\ldots,m\}$. If $uv\in E(G_{\mathsf{L}})$, i.e., $uv$ is a non-unitary facet of $\mathsf{L}$, then $f(u)\neq f(v)$, i.e., $f(u)f(v)\in K_m$. It yields that $f_|$ is a graph homomorphism of $G_{\mathsf{L}}$ to $K_m$. Therefore, $\chi(G_{\mathsf{L}})\leq m=\chi(\mathsf{L})$.  
\end{proof}

\medskip A simplicial complex $\mathsf{L}$ is said to be \textit{pure} if $\dim \mathsf{L}<\infty$ and every facet is of dimension $\dim \mathsf{L}$.

\medskip As a direct consequence of Proposition~\ref{prop:chrom-graph-chro}, we have the following result.

\begin{corollary}\label{cor:gl-l-lstar}
Let $\mathsf{L}$ be a simplicial complex.
\begin{enumerate}
    \item[(1)] If $\mathsf{L}$ is pure with dimension at least 1, then \[\chi(G_{\mathsf{L}})\leq\chi(\mathsf{L})\leq\lceil \chi(\mathsf{L}^\ast)/\dim(\mathsf{L})\rceil.\]
    \item[(2)] If $\mathsf{L}$ does not have isolated vertices and $E(G_{\mathsf{L}})\neq\emptyset$, then \begin{equation}\label{eqn:chi-chi-chi}
     \chi(G_{\mathsf{L}})\leq \chi(\mathsf{L})\leq  \chi(\mathsf{L}^\ast).
 \end{equation} 
\end{enumerate}
\end{corollary}

\medskip Example~\ref{exam:chromatic-numberex} shows that the inequalities in (\ref{eqn:chi-chi-chi}) can be equalities. Moreover, if $\mathsf{L}$ is a $1$-dimensional simplicical complex without isolated vertices, we have $G_{\mathsf{L}}=\mathsf{L}^\ast$ and thus $\chi(\mathsf{L})=\chi(\mathsf{L}^\ast)=\chi(G_{\mathsf{L}})$. 

\medskip We have the following statement, which is fundamental in Theorem~\ref{thm:lower-bound}.

\begin{proposition}\label{prop:chromatic-union}
\noindent\begin{enumerate}
    \item[(1)] If $\mathsf{L}\stackrel{\mathrm{facet}}{\to} \mathsf{K}$, then \[\chi(\mathsf{L})\leq\chi(\mathsf{K}).\]   
    \item[(2)] Let $\mathsf{L}$ be a simplicial complex, and  let $\mathsf{L}_1,\ldots,\mathsf{L}_m$ be subcomplexes of $\mathsf{L}$ such that $\mathsf{L}=\mathsf{L}_1\cup\cdots\cup \mathsf{L}_m$. Then, we have \[\chi(\mathsf{L})\leq \prod_{j=1}^m\chi(\mathsf{L}_j).\]  
\end{enumerate}  
\end{proposition}
\begin{proof}
\noindent\begin{enumerate}
    \item[(1)] Let $f:\mathsf{L}\stackrel{\mathrm{facet}}{\to} \mathsf{K}$ be a facet simplicial map. Let $m=\chi(\mathsf{K})$. Consider a surjective mapping $g:V(\mathsf{K})\to \{1,2,\ldots,m\}$ such that $g(v_i)\neq g(v_j)$ for some $i\neq j$ whenever $v_1\cdots v_n$ is a non-unitary facet of $\mathsf{K}$. Then, consider the composite mapping $g\circ f:V(\mathsf{L})\to \{1,2,\ldots,m\}$. If $u_1\cdots u_\ell$ is a non-unitary facet of $\mathsf{L}$, then $f(u_1)\cdots f(u_\ell)$ is a non-unitary facet of $\mathsf{K}$ (here we use that $f$ is a facet simplicial map). Set $f(u_1)\cdots f(u_\ell)=v_1\cdots v_k$ with $2\leq k\leq \ell$. Then there exists $i\neq j$ such that $g(v_i)\neq g(v_j)$. Let $u_r$ and $u_s$ such that $f(u_r)=v_i$ and $f(u_s)=v_j$ (of course $r\neq s$). Moreover, $(g\circ f)(u_r)\neq (g\circ f)(u_s)$. Hence, there exists a mapping $g\circ f:V(\mathsf{L})\to \{1,2,\ldots,m\}$ such that $(g\circ f)(u_r)\neq (g\circ f)(u_s)$ for some $r\neq s$ whenever $u_1\cdots u_\ell$ is a non-unitary facet of $\mathsf{L}$. Therefore, $\mathsf{L}$ admits a $n$-coloring with $n\leq m$, and it implies that $\chi(\mathsf{L})\leq n\leq \chi(\mathsf{K})$.  
     \item[(2)] Since the chromatic number does not change when we add isolated vertices, we can assume that $V(\mathsf{L}_i)=V(\mathsf{L})$ for each $i$ (i.e., each $\mathsf{L}_i$ is a spanning subcomplex of $\mathsf{L}$). Suppose that $\ell_i=\chi(\mathsf{L}_i)$, and consider for each $i$ a surjective mapping $f_i:V(\mathsf{L}_i)\to \{1,2,\ldots,\ell_i\}$ such that $f_i(v_r)\neq f_i(v_s)$ for some $r\neq s$ whenever $v_1\cdots v_n$ is a non-unitary facet of $\mathsf{L}_i$. Define the map $f:V(\mathsf{L})\to \{1,\ldots,\ell_1\}\times\cdots\times \{1,\ldots,\ell_m\}$ by \[f(v)=\left(f_1(v),\ldots,f_m(v)\right) \quad \text{ for all $v\in V(\mathsf{L})$.}\]  Since $f(v_r)\neq f(v_s)$ for some $r\neq s$ (indeed, $f_i(v_r)\neq f_i(v_s)$ for some $i$) whenever $v_1\cdots v_n$ is a non-unitary facet of $\mathsf{L}$ (and of course $v_1\cdots v_n$ is a non-unitary facet of $\mathsf{L}_i$ for some $i$), $\mathsf{L}$ admits a $n$-coloring with $n\leq \prod_{j=1}^{m}\ell_j$, and it implies that $\chi(\mathsf{L})\leq n\leq \prod_{j=1}^{m}\ell_j$. Therefore, $\chi(\mathsf{L})\leq \prod_{j=1}^m\chi(\mathsf{L}_j)$.    
     \end{enumerate} 
\end{proof}

\begin{remark}\label{rem:strict-chroma}
    We can define the following concept. The \textit{strict chromatic number} of $\mathsf{L}$, denoted by $\chi_s(\mathsf{L})$, is defined as the smallest $k$ such that there exists a strict simplicial map $\mathsf{L}\stackrel{\text{s}}{\to}\Gamma_k$. Note that if we remove or add isolated vertices from $\mathsf{L}$, its strict chromatic number does not change. Furthermore, if $\mathsf{L}\stackrel{\text{s}}{\to} \mathsf{K}$, then $\chi_s(\mathsf{L})\leq\chi_s(\mathsf{K})$. However, we observe that $\chi_s(\mathsf{L})=\chi(\mathsf{L}^\ast)$ because any map $f:V(\mathsf{L})\to\{1,\ldots,k\}$ is a strict simplicial map from $\mathsf{L}$ to $\Gamma_k$ if and only if it is a graph homomorphism from $\mathsf{L}^\ast$ to $K_k$. Hence, given a simplicial complex $\mathsf{L}$, and subcomplexes $\mathsf{L}_1,\ldots,\mathsf{L}_m$ of $\mathsf{L}$ such that $\mathsf{L}=\mathsf{L}_1\cup\cdots\cup \mathsf{L}_m$. Then \[\chi_s(\mathsf{L})\leq \prod_{j=1}^m\chi_s(\mathsf{L}_j).\]  
\end{remark}

%%%%%%%%%%%%%%%%%%%%%%%%%%%%%%%%%%%%%%%%%%%%%%%%%%%%%%%%%%%%%%%%%%%%%%%%%%%%%%%%%%%%%%%%%%%%%%%%%%%%%%%%%%%%%%%%%%%%%%%%%%%%%%%%%
\section{(Injective) facet-complexity}\label{sec:one}
In this section, we introduce the notion of (injective) facet-complexity and its properties. Several examples are provided to support this theory.

\medskip The notion of (injective) hom-complexity between graphs was recently introduced in \cite{zapata2024}. In this work, we extend this notion to higher dimensions (i.e., simplicial complexes). 

\subsection{Definitions and Examples} 
Given two simplicial complexes $\mathsf{L}$ and $\mathsf{K}$, in general, a facet or strict simplicial map $f:\mathsf{L}\to \mathsf{K}$ may not exist. In contrast, any constant map between simplicial complexes is a simplicial map. A significant challenge in simplicial complex theory is identifying  facet or strict simplicial maps. Therefore, we present the main definition of this work.

\begin{definition}[(Injective) Facet-complexity]\label{defn:complexity}
Let $\mathsf{L}$ and $\mathsf{K}$ be simplicial complexes.
    \begin{enumerate}
        \item[(1)]  The \textit{facet-complexity from $\mathsf{L}$ to $\mathsf{K}$}, denoted by $\text{C}(\mathsf{L};\mathsf{K})$, is the least positive integer $k$ such that there exist subcomplexes $\mathsf{L}_1,\ldots,\mathsf{L}_k$ of $\mathsf{L}$ satisfying  $\mathsf{L}=\mathsf{L}_1\cup\cdots\cup \mathsf{L}_k$, with the property that for each $\mathsf{L}_i$, there exists a facet simplicial map   $f_i:\mathsf{L}_i\stackrel{\mathrm{facet}}{\to} \mathsf{K}$. We set $\text{C}(\mathsf{L};\mathsf{K})=\infty$ if no such integer $k$ exists. 
        \item[(2)]  The \textit{injective facet-complexity from $\mathsf{L}$ to $\mathsf{K}$}, denoted by $\text{IC}(\mathsf{L};\mathsf{K})$, is the least positive integer $k$ such that there exist subcomplexes $\mathsf{L}_1,\ldots,\mathsf{L}_k$ of $\mathsf{L}$ satisfying $\mathsf{L}=\mathsf{L}_1\cup\cdots\cup \mathsf{L}_k$, and for each $\mathsf{L}_i$, there exists an injective facet simplicial map   $f_i:\mathsf{L}_i\stackrel{\mathrm{facet}}{\to} \mathsf{K}$. We set $\text{IC}(\mathsf{L};\mathsf{K})=\infty$ if no such integer $k$ exists.
    \end{enumerate} 
\end{definition}

\medskip Likewise, we define the \textit{strict-complexity from $\mathsf{L}$ to $\mathsf{K}$}, denoted by $\text{C}_s(\mathsf{L};\mathsf{K})$, and the \textit{injective strict-complexity from $\mathsf{L}$ to $\mathsf{K}$}, denoted by $\text{IC}_s(\mathsf{L};\mathsf{K})$. 

\medskip A collection $\mathcal{M}=\{f_i:\mathsf{L}_i\to \mathsf{K}\}_{i=1}^\ell$, where $\mathsf{L}_1,\ldots,\mathsf{L}_\ell$ are subcomplexes of $\mathsf{L}$ such that $\mathsf{L}=\mathsf{L}_1\cup\cdots\cup \mathsf{L}_\ell$ and  each $f_i:\mathsf{L}_i\to \mathsf{K}$ is a facet simplicial map, is called a \textit{quasi-facet simplicial map} from $\mathsf{L}$ to $\mathsf{K}$. A quasi-facet simplicial map $\mathcal{M}=\{f_i:\mathsf{L}_i\to \mathsf{K}\}_{i=1}^\ell$ is termed \textit{optimal} if $\ell=\text{C}(\mathsf{L};\mathsf{K})$. Observe that a unitary quasi-facet simplicial map $\{f:\mathsf{L}\to \mathsf{K}\}$ is optimal and constitutes a facet simplicial map from $\mathsf{L}$ to $\mathsf{K}$. Additionally, any quasi-facet  simplicial map $\mathcal{M}=\{f_i:\mathsf{L}_i\to \mathsf{K}\}_{i=1}^\ell$ induces a map $f:V(\mathsf{L})\to V(\mathsf{K})$ defined by $f(v)=f_i(v)$, where $i$ is the least index such that $v\in V(\mathsf{L}_i)$. Likewise, a collection $\mathcal{M}=\{f_i:\mathsf{L}_i\to \mathsf{K}\}_{i=1}^\ell$, where $\mathsf{L}_1,\ldots,\mathsf{L}_\ell$ are subcomplexes of $\mathsf{L}$ such that $\mathsf{L}=\mathsf{L}_1\cup\cdots\cup \mathsf{L}_k$ and each $f_i:\mathsf{L}_i\to \mathsf{K}$ is an injective facet simplicial map, is called an \textit{injective quasi-facet simplicial map} from $\mathsf{L}$ to $\mathsf{K}$. An injective quasi-facet simplicial map $\mathcal{M}=\{f_i:\mathsf{L}_i\to \mathsf{K}\}_{i=1}^\ell$ is termed \textit{optimal} if $\ell=\text{IC}(\mathsf{L};\mathsf{K})$.

\medskip Note that if we remove or add isolated vertices from $\mathsf{L}$, its facet-complexity $\text{C}(\mathsf{L};\mathsf{K})$ does not change. This statement does not hold for the injective facet-complexity. For example, $\mathrm{IC}(\mathsf{K}_3\sqcup \{\ast\};\mathsf{K}_3)=2$, whereas $\mathrm{IC}(\mathsf{K}_3;\mathsf{K}_3)=1$.

\medskip We say that a subcomplex $\mathsf{K}$ of $\mathsf{L}$ is a \textit{facet subcomplex} if any facet of $\mathsf{K}$ is a facet of $\mathsf{L}$. Hence, the inclusion $\mathsf{K}\hookrightarrow \mathsf{L}$ is an injective facet simplicial map.

\medskip By Definition~\ref{defn:complexity}, we can also make the following remark. 

\begin{remark}\label{rem:def-ob}
\noindent \begin{itemize}
    \item[(1)]  $\text{C}(\mathsf{L};\mathsf{K})\leq \text{IC}(\mathsf{L};\mathsf{K})$ for any simplicial complexes $\mathsf{L}$ and $\mathsf{K}$, since any injective quasi-facet simplicial map is a quasi-facet simplicial map. 
    \item[(2)]  $\text{C}(\mathsf{L};\mathsf{K})=1$ if and only if there exists a facet simplicial map $\mathsf{L}\to \mathsf{K}$ (i.e., $\mathsf{L}$ is facet $\mathsf{K}$-colourable). Additionally, $\text{IC}(\mathsf{L};\mathsf{K})=1$  if and only if there exists an injective facet simplicial map  $\mathsf{L}\to \mathsf{K}$, which is equivalent to saying that $\mathsf{K}$ admits a copy of $\mathsf{L}$ as a facet subcomplex.
    \item[(3)] The facet-complexity $\text{C}(\mathsf{L};\mathsf{K})$ coincides with the least positive integer $k$ such that there exist subcomplexes $\mathsf{L}_1,\ldots,\mathsf{L}_k$ of $\mathsf{L}$ satisfying  $\mathsf{L}=\mathsf{L}_1\cup\cdots\cup \mathsf{L}_k$, and each $\mathsf{L}_i$ is facet $\mathsf{K}$-colourable. 
    \item[(4)] Since facet $\mathsf{K}$-colorability does not depend on isolated vertices, we have that $\text{C}(\mathsf{L};\mathsf{K})$ coincides with the least positive integer $k$ such that there exist spanning subcomplexes $\mathsf{L}_1,\ldots,\mathsf{L}_k$ of $\mathsf{L}$ satisfying $\mathsf{L}=\mathsf{L}_1\cup\cdots\cup \mathsf{L}_k$, and each $\mathsf{L}_i$ is facet $\mathsf{K}$-colourable. 
    \item[(5)] If $\text{C}_s(\mathsf{L};\mathsf{K})<\infty$, then $\mathsf{K}$ admits a simplex of dimension $d$ whenever $\mathsf{L}$ admits a simplex of dimension $d$. In particular, \[\dim(\mathsf{L})\leq \dim(\mathsf{K}).\] Similarly, if $\text{C}(\mathsf{L};\mathsf{K})<\infty$, then $\mathsf{K}$ admits a non-unitary facet of dimension $d'\leq d$ whenever $\mathsf{L}$ admits a non-unitary facet of dimension $d$. In particular, \[\dim(\mathsf{L})\leq \dim(\mathsf{K}).\]
    \item[(6)] Note that any injective facet simplicial map is an injective simplicial map (and of course it is an injective strict simplicial map). Hence, we have \[\text{IC}(\mathsf{L};\mathsf{K})\geq\text{IC}_s(\mathsf{L};\mathsf{K}). \] 
\end{itemize}
\end{remark}

Given a simplicial complex $\mathsf{L}$, recall that $\mathsf{L}^\ast$ denotes the underlying graph of $\mathsf{L}$ (see Definition~\ref{defn:associated-graph}). More generally, for each $q\geq 0$, the \textit{$q$-skeleton} of $\mathsf{L}$ is given by \[\mathsf{L}^{(q)}=\bigcup_{d=0}^{q}F_d(\mathsf{L}).\] Note that each $q$-skeleton $\mathsf{L}^{(q)}$ is a subcomplex of $\mathsf{L}$. For instance, the $1$-skeleton $\mathsf{L}^{(1)}$ corresponds to $\mathsf{L}^\ast$. 

\medskip The symbol ($\text{IC}(\mathsf{L}^{\ast};\mathsf{K}^{\ast})$, respectively) $\text{C}(\mathsf{L}^{\ast};\mathsf{K}^{\ast})$ denotes the (injective, respectively) hom-complexity of $\mathsf{L}^{\ast}$ to $\mathsf{K}^{\ast}$ introduced in \cite{zapata2024}. That is, ($\text{IC}(\mathsf{L}^{\ast};\mathsf{K}^{\ast})$, respectively) $\text{C}(\mathsf{L}^{\ast};\mathsf{K}^{\ast})$ is the least positive integer $k$ such that there exist subgraphs $G_1,\ldots,G_k$ of $\mathsf{L}^{\ast}$ satisfying  $\mathsf{L}^{\ast}=G_1\cup\cdots\cup G_k$, with the property that for each $G_i$, there exists a (injective, respectively) graph homomorphism $f_i:G_i\to \mathsf{K}^{\ast}$.

\medskip The following remark says that the (injective) facet-complexity between the $1$-skeletons recovers the complexity between graphs.

\begin{remark}
 Let $\mathsf{L}$ and $\mathsf{K}$ be simplicial complexes.
 \begin{enumerate}
     \item[(1)] The following equalities hold: \[\text{C}_s(\mathsf{L}^{(1)};\mathsf{K}^{(1)})=\text{C}(\mathsf{L}^{\ast};\mathsf{K}^{\ast}) \quad\text{ and }\quad \text{IC}_s(\mathsf{L}^{(1)};\mathsf{K}^{(1)})=\text{IC}(\mathsf{L}^{\ast};\mathsf{K}^{\ast}).\]
     \item[(2)] Suppose that $\mathsf{K}$ has not isolated vertices (and of course $\mathsf{K}^{\ast}$ has not isolated vertices). Then \[\text{C}(\mathsf{L}^{(1)};\mathsf{K}^{(1)})=\text{C}(\mathsf{L}^{\ast};\mathsf{K}^{\ast}).\]
     \item[(3)] The equality \[\text{IC}(\mathsf{L}^{(1)};\mathsf{K}^{(1)})=\text{IC}(\mathsf{L}^{\ast};\mathsf{K}^{\ast})\] always holds.
 \end{enumerate}
\end{remark}

\medskip From Definition~\ref{defn:complexity} we have the following remark.

\begin{remark}\label{rem:ineq-complex-q}
Let $\mathsf{L}$ and $\mathsf{K}$ be simplicial complexes. \begin{enumerate}
    \item[(1)] We have \[\text{C}_s(\mathsf{L};\mathsf{K})\geq\cdots\geq \text{C}_s(\mathsf{L}^{(2)};\mathsf{K}^{(2)})\geq\text{C}(\mathsf{L}^{\ast};\mathsf{K}^{\ast})\] and \[\text{IC}_s(\mathsf{L};\mathsf{K})\geq\cdots\geq \text{IC}_s(\mathsf{L}^{(2)};\mathsf{K}^{(2)})\geq\text{IC}(\mathsf{L}^{\ast};\mathsf{K}^{\ast}).\]
    \item[(2)] Since any (injective, respectively) graph homomorphism $f:\mathsf{L}^{\ast}\to K_n$ is a (injective, respectively) strict simplicial map $f:\mathsf{L}\to\Gamma_n$ (recall that $(\Gamma_n)^\ast=K_n$), we have \[\text{C}_s(\mathsf{L};\Gamma_n)=\cdots= \text{C}_s(\mathsf{L}^{(2)};(\Gamma_n)^{(2)})=\text{C}(\mathsf{L}^{\ast};K_n) \] and \[\text{IC}_s(\mathsf{L};\Gamma_n)=\cdots= \text{IC}_s(\mathsf{L}^{(2)};(\Gamma_n)^{(2)})=\text{IC}(\mathsf{L}^{\ast};K_n).\]
    \item[(3)] Since any (injective, respectively) graph homomorphism $f:\mathsf{L}^{\ast}\to K_n$ is a (injective, respectively) strict simplicial map $f:\mathsf{L}\to\mathsf{K}_n$ whenever $\dim(\mathsf{L})\leq n-2$ (note that $(\mathsf{K}_n)^\ast=K_n$), we have \[\text{C}_s(\mathsf{L};\mathsf{K}_n)=\cdots= \text{C}_s(\mathsf{L}^{(2)};(\mathsf{K}_n)^{(2)})=\text{C}(\mathsf{L}^{\ast};K_n)\] and \[\text{IC}_s(\mathsf{L};\mathsf{K}_n)=\cdots= \text{IC}_s(\mathsf{L}^{(2)};(\mathsf{K}_n)^{(2)})=\text{IC}(\mathsf{L}^{\ast};K_n)\] whenever $\dim(\mathsf{L})\leq n-2$.
\end{enumerate}    
\end{remark}

\medskip The following example demonstrates that the facet-complexity can be strictly more than the facet-complexity between the $1$-skeletons.

\begin{example}\label{exam:complexity-less-injective}
    Let $\mathsf{L}$ and $\mathsf{K}$ be simplicial complexes as given in Example~\ref{exam:chromatic-numberex}, that is, they are defined as follows: $V(\mathsf{L})=\{a,b,c,d,e\}$, $F_1(\mathsf{L})=\{ab,bc,ac,cd,de,ce\}$, $F_2(\mathsf{L})=\{abc\}$, and  $V(\mathsf{K})=\{a',b',c',d'\}$, $F_1(\mathsf{K})=\{a'b',b'c',c'a',c'd'\}$, $F_2(\mathsf{K})=\{a'b'c'\}$. 
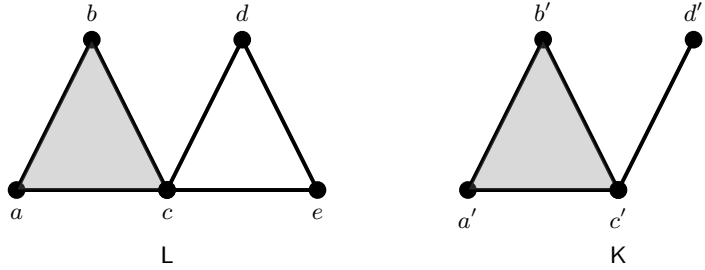
    $$
\begin{tikzpicture}
% Simple L
\Vertex[x=0,y=0,size=0.2,label=$a$,position=below,color=black]{A} 
\Vertex[x=1, y=2, size=0.2,label=$b$,position=above,color=black]{B}
\Vertex[x=2, y=0, size=0.2,label=$c$,position=below,color=black]{C} 
\Vertex[x=3, y=2, size=0.2,label=$d$,position=above,color=black]{D}
\Vertex[x=4, y=0, size=0.2,label=$e$,position=below,color=black]{E} 
% Filling the triangle A, B, C with a color
 \fill[black!50, draw = black, opacity=0.3] (0,0) -- (1,2) -- (2,0) -- cycle;
\Edge[color=black](A)(B) 
\Edge[color=black](B)(C)
\Edge[color=black](C)(A)
\Edge[color=black](C)(D)
\Edge[color=black](C)(E)
\Edge[color=black](D)(E)
\Vertex[x=2,y=0,size=0.2,distance=0.5cm,label=$\mathsf{L}$,position=below,color=black]{M} 
%SIMPLEX K
\Vertex[x=6,y=0,size=0.2,label=$a'$,position=below,color=black]{F} 
\Vertex[x=7,y=2,size=0.2,label=$b'$,position=above,color=black]{G} 
\Vertex[x=8,y=0,size=0.2,label=$c'$,position=below,color=black]{H} 
\Vertex[x=9,y=2,size=0.2,label=$d'$,position=above,color=black]{I} 
% Filling the triangle A', B', C' with a color
 \fill[black!50, draw = black, opacity=0.3] (6,0) -- (7,2) -- (8,0) -- cycle;
 \Edge[color=black](F)(G) 
\Edge[color=black](G)(H)
\Edge[color=black](H)(F)
\Edge[color=black](H)(I)
\Vertex[x=8, y=0, size=0.2,distance=0.5cm,label=$\mathsf{K}$,position=below,color=black]{N}
  \end{tikzpicture}
 $$ Note that the map $f:V(\mathsf{L})\to V(\mathsf{K})$ defined by $f(a)=f(e)=a'$, $f(b)=f(d)=b'$, and $f(c)=c'$ is a graph homomorphism from $\mathsf{L}^\ast$ to $\mathsf{K}^\ast$. Hence, we have $\mathrm{C}(\mathsf{L}^\ast;\mathsf{K}^\ast)=1$.
 
On the other hand, since $\chi(\mathsf{L})=3$ and $\chi(\mathsf{K})=2$, we have that there is no facet simplicial map of $\mathsf{L}$ to $\mathsf{K}$ (by Proposition~\ref{prop:chromatic-union}). Hence, $\mathrm{C}(\mathsf{L};\mathsf{K})\geq 2$. Additionally, consider the subcomplexes $\mathsf{L}_1$ and $\mathsf{L}_2$ of $L$, defined as follows:\begin{align*}
     V(\mathsf{L}_1)&=\{a,b,c,d,e\},\\
     F_1(\mathsf{L}_1)&=\{ab,bc,ac,cd,de\},\\
     F_2(\mathsf{L}_1)&=\{abc\},\\
     V(\mathsf{L}_2)&=\{c,e\},\\
     F_1(\mathsf{L}_2)&=\{ce\}.
 \end{align*} Together with the facet simplicial maps $f_1:\mathsf{L}_1\stackrel{\text{facet}}{\to} \mathsf{K}$ and $f_2:\mathsf{L}_2\stackrel{\text{facet}}{\to} \mathsf{K}$, defined by: \begin{align*}
     f_1(a)&=a',\\
     f_1(b)&=b',\\
     f_1(c)&=f_1(e)=c',\\
     f_1(d)&=d',\\
     f_2(c)&=c',\\
     f_2(e)&=d'.\\
 \end{align*} Note that $\mathsf{L}=\mathsf{L}_1\cup \mathsf{L}_2$. Thus, we have $\mathrm{C}(\mathsf{L};\mathsf{K})\leq 2$. Therefore, we conclude that $\mathrm{C}(\mathsf{L};\mathsf{K})=2$. We left to the reader to check the equality $\mathrm{IC}(\mathsf{L};\mathsf{K})=3$.
\end{example}
%%%%%%%%%%%%%%%%%%%%%%%%%%%%%%%%%%%%%%%%%%%%%%%%%%%%%%%%%%%%%%%%%%%%%%%%%%%%%%%%%%%%%%%%%%%%%%%%%%%%%%%%%%%%%%%%%%%%%%%%%%%%%%%%%%%%%%%%%%%%%%%%%%%%%%%%%%%%%%%%%%%%%%%%%%%%%%%%%%%%%%%%%%%%%%%%%%%%%%%%%%%%%%%%%%%%%%%%%%%%%%%%

\subsection{Triangular Inequality} Given a facet simplicial map $f:\mathsf{L}\to \mathsf{H}$ and a subcomplex $\mathsf{K}$ of $\mathsf{H}$, the \textit{image inverse} of $\mathsf{K}$ through $f$ is the subcomplex $f^{-1}(\mathsf{K})$, defined as follows:
 \begin{itemize}
     \item The vertex set is given by $V(f^{-1}(\mathsf{K}))=f^{-1}(V(\mathsf{K}))$.
     \item A subset $F\subseteq f^{-1}(V(\mathsf{K}))$ is a simplex of $f^{-1}(K)$ if and only if $F\in \mathsf{L}$ and $f(F)\in \mathsf{K}$. 
 \end{itemize} Note that the restriction map $f_|:f^{-1}(V(\mathsf{K}))\to V(\mathsf{K})$ is a surjective simplicial map from the subcomplex $f^{-1}(\mathsf{K})$ to $\mathsf{K}$, called the \textit{restriction simplicial map}, and is denoted by $f_|:f^{-1}(\mathsf{K})\to \mathsf{K}$. It is not a facet simplicial map in general. However, we have the following remark.

\begin{remark}\label{rem:restriction-facet}
Let $f:\mathsf{L}\to \mathsf{H}$ be a facet simplicial map and $\mathsf{K}$ be a subcomplex of $\mathsf{H}$.
\begin{enumerate}
    \item[(1)] Observe that $f_|:f^{-1}(\mathsf{K})\to \mathsf{K}$ is a facet simplicial map whenever any facet of $f^{-1}(\mathsf{K})$ is a facet of $\mathsf{L}$. Note that if $F$ is a facet of $\mathsf{H}$ and $F\in\mathsf{K}$, then $F$ is a facet of $\mathsf{K}$.
    \item[(2)] Any facet of $f^{-1}(\mathsf{K})$ is a facet of $\mathsf{L}$ whenever any facet of $\mathsf{K}$ is a facet of $\mathsf{H}$. In fact, suppose that $F$ is a facet of $f^{-1}(\mathsf{K})$. We will check that $F$ is a facet of $\mathsf{L}$. By contradiction, suppose that $F$ is not a facet of $\mathsf{L}$, i.e., there exists a facet $G$ of $\mathsf{L}$ such that $F\subseteq G$ and $G\not\subseteq f^{-1}(V(\mathsf{K}))$. Then, $f(G)$ is a facet of $\mathsf{H}$. Since any facet of $\mathsf{K}$ is a facet of $\mathsf{H}$, we have $f(G)\subseteq K$, which is a contradiction to the statement $G\not\subseteq f^{-1}(V(\mathsf{K}))$.  
\end{enumerate}
\end{remark}
  
\medskip Given three simplicial complexes $\mathsf{L}, \mathsf{H}$, and $\mathsf{K}$, there is a relation between the hom-complexities $\mathrm{C}(\mathsf{L};\mathsf{H}), \mathrm{C}(\mathsf{H};\mathsf{K})$, and $\mathrm{C}(\mathsf{L};\mathsf{K})$. Likewise, the same holds for the injective facet-complexity. 

\begin{theorem}[Triangular Inequality]\label{thm:inequality-three-graphs}
  Let $\mathsf{L}, \mathsf{H}$, and $\mathsf{K}$ be simplicial complexes. Then, \[\mathrm{C}(\mathsf{L};\mathsf{K})\leq \mathrm{C}(\mathsf{L};\mathsf{H})\cdot \mathrm{C}(\mathsf{H};\mathsf{K}) \quad \text{ and } \quad \mathrm{IC}(\mathsf{L};\mathsf{K})\leq \mathrm{IC}(\mathsf{L};\mathsf{H})\cdot \mathrm{IC}(\mathsf{H};\mathsf{K}).\]  
\end{theorem}
\begin{proof}
  Let $m=\mathrm{C}(\mathsf{L};\mathsf{H})$ and $n=\mathrm{C}(\mathsf{H};\mathsf{K})$. Let $\mathcal{M}_1=\{g_{i}:\mathsf{L}_{i}\to \mathsf{H}\}_{i=1}^{m}$ be an optimal quasi-facet simplicial map from $\mathsf{L}$ to $\mathsf{H}$, and $\mathcal{M}_2=\{h_{j}:\mathsf{H}_{j}\to \mathsf{K}\}_{j=1}^{n}$ be an optimal quasi-facet simplicial map from $\mathsf{H}$ to $\mathsf{K}$. Define  $\mathsf{L}_{i,j}:=g_{i}^{-1}(\mathsf{H}_j)$ for each $i\in\{1,\ldots,m\}$ and each $j\in\{1,\ldots,n\}$ (noting that some $\mathsf{L}_{i,j}$ may be empty). We have $\mathsf{L}=\bigcup_{i,j=1}^{m,n} \mathsf{L}_{i,j}$. Observe that $\mathsf{L}_{i,j}$ is a subcomplex of $\mathsf{L}_i$ (and consequently a subcomplex of $\mathsf{L}$). If $\mathsf{L}_{i,j}\neq\emptyset$, we also consider the restriction simplicial map $(g_i)_|:\mathsf{L}_{i,j}\to \mathsf{H}_j$. This leads to the composition \[\mathsf{L}_{i,j}\stackrel{(g_i)_|}{\to} \mathsf{H}_j\stackrel{h_j}{\to} \mathsf{K}.\] Without leaving the generality, we can suppose that any facet of $\mathsf{H}_j$ is a facet of $\mathsf{H}$. Then, by Remark~\ref{rem:restriction-facet},  each restriction $(g_i)_|$ is a facet simplicial map. Hence, each composition $h_j\circ (g_i)_|$ is a facet simplicial map. Therefore, we obtain $\mathrm{C}(\mathsf{L};\mathsf{K})\leq m\cdot n=\mathrm{C}(\mathsf{L};\mathsf{H})\cdot \mathrm{C}(\mathsf{H};\mathsf{K})$.  

  Likewise, we obtain the inequality $\mathrm{IC}(\mathsf{L};\mathsf{K})\leq \mathrm{IC}(\mathsf{L};\mathsf{H})\cdot \mathrm{IC}(\mathsf{H};\mathsf{K})$ because if $g_{i}$ and $h_{j}$ are injective, then the composition $\mathsf{L}_{i,j}\stackrel{(g_i)_|}{\to} \mathsf{H}_j\stackrel{h_j}{\to} \mathsf{K}$ is also injective. 
\end{proof}

The inequality in Theorem~\ref{thm:inequality-three-graphs} is sharp. For instance, consider $\mathsf{K}=\mathsf{H}$; then $\mathrm{C}(\mathsf{L};\mathsf{H})= \mathrm{C}(\mathsf{L};\mathsf{H})\cdot \mathrm{C}(\mathsf{H};\mathsf{H})$. 

%%%%%%%%%%%%%%%%%%%%%%%%%%%%%%%%%%%%%%%%%%%%%%%%%%%%%%%%%%%%%%%%%%%%%%%%%%%%%%%%%%%%%%%%%%%%%%%%%%%%%%%%%%%%%%%%%%%%

\subsection{Simplicial complex invariant}
The following result demonstrates that the existence of a facet simplicial map implies inequalities between the (injective) hom-complexities. 

\begin{theorem}\label{prop:complexity-subgraphs}
    Let $\mathsf{L}'\to \mathsf{L}$ and $\mathsf{H}'\to \mathsf{H}$ be facet simplicial maps.  
    \begin{itemize}
        \item[(1)] We have: \[\mathrm{C}(\mathsf{L}';\mathsf{H})\leq \mathrm{C}(\mathsf{L};\mathsf{H})\leq \mathrm{C}(\mathsf{L};\mathsf{H}').\] 
        \item[(2)] Moreover, if $\mathsf{L}'\to \mathsf{L}$ and $\mathsf{H}'\to \mathsf{H}$ are injective, then \[\mathrm{IC}(\mathsf{L}';\mathsf{H})\leq \mathrm{IC}(\mathsf{L};\mathsf{H})\leq \mathrm{IC}(\mathsf{L};\mathsf{H}').\]
    \end{itemize}
\end{theorem}
\begin{proof}
 It follows as a direct application of the triangular inequality (Theorem~\ref{thm:inequality-three-graphs}). 
\end{proof}

From Theorem~\ref{prop:complexity-subgraphs}(1), we observe that if $\mathsf{L}'\stackrel{\mathrm{facet}}{\to} \mathsf{L}$ and $\mathsf{L}\stackrel{\mathrm{facet}}{\to} \mathsf{L}'$, then $\mathrm{C}(\mathsf{L}';\mathsf{H})=\mathrm{C}(\mathsf{L};\mathsf{H})$ for any simplicial complex $\mathsf{H}$. Similarly, if $\mathsf{H}'\stackrel{\mathrm{facet}}{\to} \mathsf{H}$ and $\mathsf{H}\stackrel{\mathrm{facet}}{\to} \mathsf{H}'$, then $\mathrm{C}(\mathsf{L};\mathsf{H}')=\mathrm{C}(\mathsf{L};\mathsf{H})$ for any simplicial complex $\mathsf{L}$. In particular, this shows that (injective) facet-complexity is a simplicial complex invariant, meaning it is preserved under isomorphisms.

\begin{corollary}[Simplicial Complex Invariant]\label{cor:invariant-iso-complexity}
    If $\mathsf{L}'$ is isomorphic to $\mathsf{L}$ and $\mathsf{H}'$ is isomorphic to $\mathsf{H}$, then \[\mathrm{C}(\mathsf{L};\mathsf{H})=\mathrm{C}(\mathsf{L}';\mathsf{H}') \quad \text{ and } \quad \mathrm{IC}(\mathsf{L};\mathsf{H})=\mathrm{IC}(\mathsf{L}';\mathsf{H}').\]
\end{corollary}

Furthermore, Theorem~\ref{prop:complexity-subgraphs} implies that facet-complexity provides a numerical obstruction to the existence of a facet  simplicial map.  

\begin{proposition}\label{prop:no-existencia}
\noindent\begin{enumerate}
    \item[(1)]  Let $\mathsf{L}$ and $\mathsf{L}'$ be simplicial complexes. We have:
    \begin{enumerate}
        \item[(i)]  If $\mathrm{C}(\mathsf{L}';\mathsf{H})>\mathrm{C}(\mathsf{L};\mathsf{H})$ for some simplicial complex $\mathsf{H}$, then $\mathsf{L}'\stackrel{\mathrm{facet}}{\not\to} \mathsf{L}$. 
         \item[(ii)]  If $\mathrm{IC}(\mathsf{L}';\mathsf{H})>\mathrm{IC}(\mathsf{L};\mathsf{H})$ for some simplicial complex $\mathsf{H}$, then there is no injective facet simplicial map from $\mathsf{L}'$ to  $\mathsf{L}$. 
    \end{enumerate}
     \item[(2)]  Let $\mathsf{H}$ and $\mathsf{H}'$ be simplicial complexes. We have:
    \begin{enumerate}
        \item[(i)]  If $\mathrm{C}(\mathsf{L};\mathsf{H})>\mathrm{C}(\mathsf{L};\mathsf{H}')$ for some simplicial complex $\mathsf{L}$, then $\mathsf{H}'\stackrel{\mathrm{facet}}{\not\to} \mathsf{H}$. 
         \item[(ii)] If $\mathrm{IC}(\mathsf{L};\mathsf{H})>\mathrm{IC}(\mathsf{L};\mathsf{H}')$ for some simplicial complex $\mathsf{L}$, then there is no injective facet simplicial map from $\mathsf{L}'$ to  $\mathsf{H}$. 
    \end{enumerate}
\end{enumerate}    
\end{proposition}
\begin{proof}
   It is sufficient to use the contrapositive of each implication in Theorem~\ref{prop:complexity-subgraphs}.
\end{proof}

%%%%%%%%%%%%%%%%%%%%%%%%%%%%%%%%%%%%%%%%%%%%%%%%%%%%%%%%%%%%%%%%%%%%%%%%%%%%%%%%%%%%%%%%%%%%%%%%%%%%%%%%%%%
\subsection{Sub-additivity} 
 The following statement demonstrates the sub-additivity property of (injective) facet-complexity. 

\begin{theorem}[Sub-additivity]\label{thm:category-union}
    Let $\mathsf{L},\mathsf{H}$ be simplicial complexes, and let $\mathsf{A},\mathsf{B}$ be facet subcomplexes of $\mathsf{L}$ such that $\mathsf{L}=\mathsf{A}\cup \mathsf{B}$. Then: \begin{itemize}
        \item[(1)] $\max\{\mathrm{C}(\mathsf{A};\mathsf{H}),\mathrm{C}(\mathsf{B};\mathsf{H})\}\leq \mathrm{C}(\mathsf{L};\mathsf{H})\leq \mathrm{C}(\mathsf{A};\mathsf{H})+\mathrm{C}(\mathsf{B};\mathsf{H}).$ 
         \item[(2)] $\max\{\mathrm{IC}(\mathsf{A};\mathsf{H}),\mathrm{IC}(\mathsf{B};\mathsf{H})\}\leq \mathrm{IC}(\mathsf{L};\mathsf{H})\leq \mathrm{IC}(\mathsf{A};\mathsf{H})+\mathrm{IC}(\mathsf{B};\mathsf{H}).$ 
    \end{itemize} 
\end{theorem}
\begin{proof}
\noindent \begin{itemize}
        \item[(1)] The inequality $\max\{\mathrm{C}(\mathsf{A};\mathsf{H}),\mathrm{C}(\mathsf{B};\mathsf{H})\}\leq \mathrm{C}(\mathsf{L};\mathsf{H})$ follows from Theorem~\ref{prop:complexity-subgraphs}(1), applied to the inclusions $\mathsf{A}\hookrightarrow \mathsf{L}$ and $\mathsf{B}\hookrightarrow \mathsf{L}$. To demonstrate the other inequality, suppose that $\mathrm{C}(\mathsf{A};\mathsf{H})=m$ and $\mathrm{C}(\mathsf{B};\mathsf{H})=k$. Let $\{f_i:\mathsf{A}_i\to \mathsf{H}\}_{i=1}^{m}$ be an optimal quasi-facet simplicial map from $\mathsf{A}$ to $\mathsf{H}$, and $\{g_j:\mathsf{B}_j\to \mathsf{H}\}_{j=1}^{k}$ be an optimal quasi-facet simplicial map from $\mathsf{B}$ to $\mathsf{H}$. The combined collection $\{f_1:\mathsf{A}_1\to \mathsf{H},\ldots,f_m:\mathsf{A}_m\to \mathsf{H},g_1:\mathsf{B}_1\to \mathsf{H},\ldots,g_k:\mathsf{B}_k\to \mathsf{H}\}$ forms a quasi-facet simplicial map from $\mathsf{L}$ to $\mathsf{H}$. Therefore, we have $\mathrm{C}(\mathsf{L};\mathsf{H})\leq m+k=\mathrm{C}(\mathsf{A};\mathsf{H})+\mathrm{C}(\mathsf{B};\mathsf{H})$. 

        This completes the proof of the sub-additivity of facet-complexity.

        \item[(2)] Likewise, we obtain the sub-additivity of injective facet-complexity.
    \end{itemize}  
\end{proof}

Theorem~\ref{thm:category-union} implies the following corollary:

\begin{corollary}\label{cor:linear-complexity}
    Let $\mathsf{L}$ and $\mathsf{H}$ be simplicial complexes, and $\mathsf{A}$ and $\mathsf{T}$ be facet subcomplexes of $\mathsf{L}$ such that $\mathsf{L}=\mathsf{A}\cup \mathsf{T}$. Then:
    \begin{enumerate}
        \item[(1)] If $\mathrm{C}(\mathsf{T};\mathsf{H})=1$, then \[\mathrm{C}(\mathsf{A};\mathsf{H})\leq \mathrm{C}(\mathsf{L};\mathsf{H})\leq \mathrm{C}(\mathsf{A};\mathsf{H})+1.\] 
        \item[(2)] If $\mathrm{IC}(\mathsf{T};\mathsf{H})=1$, then \[\mathrm{IC}(\mathsf{A};\mathsf{H})\leq \mathrm{IC}(\mathsf{L};\mathsf{H})\leq \mathrm{IC}(\mathsf{A};\mathsf{H})+1.\] 
    \end{enumerate}
\end{corollary}

The following result shows that the first inequality of Theorem~\ref{thm:category-union}(1) can be an equality. 

\begin{proposition}\label{prop:complexity-disjoint-union}
  Let $\mathsf{L}$ be a simplicial complex, and let $\mathsf{A}$ and $\mathsf{B}$ be facet subcomplexes of $\mathsf{L}$ such that $V(\mathsf{A})\cap V(\mathsf{B})=\emptyset$ and $\mathsf{L}=\mathsf{A}\sqcup \mathsf{B}$ (see Definition~\ref{defn:union-graphs}(2)). Then, for any simplicial complex $\mathsf{H}$, we have \[\mathrm{C}(\mathsf{L};\mathsf{H})=\max\{\mathrm{C}(\mathsf{A};\mathsf{H}),\mathrm{C}(\mathsf{B},\mathsf{H})\}.\]   
\end{proposition}
\begin{proof}
 The inequality $\max\{\mathrm{C}(\mathsf{A};\mathsf{H}),\mathrm{C}(\mathsf{B},\mathsf{H})\}\leq \mathrm{C}(\mathsf{L};\mathsf{H})$ follows from Theorem~\ref{thm:category-union}(1). We will now verify the inequality $\mathrm{C}(\mathsf{L};\mathsf{H})\leq m$, where $m=\max\{\mathrm{C}(\mathsf{A};\mathsf{H}),\mathrm{C}(\mathsf{B},\mathsf{H})\}$. In fact, let $\{f_i:\mathsf{A}_i\to \mathsf{H}\}_{i=1}^{m}$ and $\{g_i:\mathsf{B}_i\to \mathsf{H}\}_{i=1}^{m}$ be quasi-facet simplicial maps from $\mathsf{A}$ to $\mathsf{H}$ and from $\mathsf{B}$ to $\mathsf{H}$, respectively. Note that the collection $\{f_i\sqcup g_i:\mathsf{A}_i\sqcup \mathsf{B}_i\to \mathsf{H}\}_{i=1}^{m}$ forms a quasi-facet simplicial map from $\mathsf{L}$ to $\mathsf{H}$. Therefore, we have $\mathrm{C}(\mathsf{L};\mathsf{H})\leq m=\max\{\mathrm{C}(\mathsf{A};\mathsf{H}),\mathrm{C}(\mathsf{B},\mathsf{H})\}$.
\end{proof}

Observe that the condition $V(\mathsf{A})\cap V(\mathsf{B})=\emptyset$ in Proposition~\ref{prop:complexity-disjoint-union} cannot be removed. To illustrate this, consider the $1$-dimensional simplicial complexes $\mathsf{A}$ and $\mathsf{B}$ defined as follows:  $V(\mathsf{A})=\{1,2,3\}$, $F_1(\mathsf{A})=\{12,13\}$, $V(\mathsf{B})=\{2,3\}$ and $F_1(\mathsf{B})=\{23\}$. Note that, $\mathrm{C}(\mathsf{A};\Gamma_2)=1$ and $\mathrm{C}(\mathsf{B};\Gamma_2)=1$. However, since $\mathsf{A}\cup \mathsf{B}=\mathsf{K}_3$, we have $\mathrm{C}(\mathsf{A}\cup \mathsf{B};\mathsf{K}_2)=2$.

%%%%%%%%%%%%%%%%%%%%%%%%%%%%%%%%%%%%%%%%%%%%%%%%%%%%%%%%%%%%%%%%%%%%%%%%%%%%%%%%%%%%%%%%%%%%%%%%%%%%%%%%%%%

\subsection{Lower bound} We have the following lower bound for facet-complexity.

\begin{theorem}[Lower Bound]\label{thm:lower-bound}
 Let $\mathbf{L}$ and $\mathbf{K}$ be simplicial complexes. \begin{enumerate}
     \item[(1)] The inequality \[\chi(\mathbf{L})\leq\chi(\mathbf{K})^{\mathrm{C}(\mathbf{L};\mathbf{K})}\] holds. Equivalently, $\log_{\chi(\mathbf{K})}\chi(\mathbf{L})\leq \mathrm{C}(\mathbf{L};\mathbf{K})$.    
     \item[(2)] Suppose that $\mathbf{K}$ has no isolated vertices. We have \[\mathrm{C}(G_\mathbf{L};G_\mathbf{K})\leq \mathrm{C}(\mathbf{L};\mathbf{K}).\]
 \end{enumerate} 
\end{theorem}
\begin{proof}
 Suppose that $m=\mathrm{C}(\mathbf{L};\mathbf{K})$ and consider $\mathbf{L}_1,\ldots,\mathbf{L}_m$ subcomplexes of $\mathbf{L}$ such that $\mathbf{L}=\mathbf{L}_1\cup\cdots\cup \mathbf{L}_m$, with a facet simplicial map $f_j:\mathbf{L}_j\stackrel{\mathrm{facet}}{\to} \mathbf{K}$ for each $\mathbf{L}_j$.
 \begin{enumerate}
     \item[(1)] Note that $\chi(\mathbf{L}_j)\leq\chi(\mathbf{K})$, see Proposition~\ref{prop:chromatic-union}(1). Then, by Proposition~\ref{prop:chromatic-union}(2), we have: \begin{align*}
        \chi(\mathbf{L})&\leq \prod_{j=1}^m\chi(\mathbf{L}_j)\\
        &\leq \prod_{j=1}^m\chi(\mathbf{K})\\
        &=\chi(\mathbf{K})^m.
    \end{align*} 
    \item[(2)] For each $j=1,\ldots,m$, set the graph $G_j$ given by \begin{align*}
        V(G_j)&= V(G_\mathbf{L})\cap V(\mathbf{L}_j),\\
        E(E_j)&= E(G_\mathbf{L})\cap \mathbf{L}_j.
    \end{align*} Note that each $G_j$ is a subgraph of $G_\mathbf{L}$. Furthermore, $G_\mathbf{L}=G_1\cup\cdots\cup G_m$. In addition, each facet simplicial map $f_j:\mathbf{L}_j\to \mathbf{K}$ restricts to a graph homomorphism $(f_j)_|:G_j\to G_\mathbf{K}$ (recall that $\mathbf{K}$ has no isolated vertices). Therefore, $\mathrm{C}(G_\mathbf{L};G_\mathbf{K})\leq m=\mathrm{C}(\mathbf{L};\mathbf{K})$.
    \end{enumerate} 
\end{proof}

Theorem~\ref{thm:lower-bound} implies the following statement.

\begin{corollary}\label{cor:chro-lower-bound}
Let $m\geq 1$ be an integer. Let $\mathbf{L}$ and $\mathbf{K}$ be simplicial complexes such that $\chi(\mathbf{L})\geq 2$. If $\chi(\mathbf{K})^{m-1}+1\leq \chi(\mathbf{L})$, then $\mathrm{C}(\mathbf{L};\mathbf{K})\geq m$.   
\end{corollary}
\begin{proof}
 The case $m=1$ is straightforward. Let $m\geq 2$. Suppose that $\mathrm{C}(\mathbf{L};\mathbf{K})\leq m-1$. By Theorem~\ref{thm:lower-bound}, we have \begin{align*}
    \chi(\mathbf{L})&\leq \chi(\mathbf{K})^{\mathrm{C}(\mathbf{L};\mathbf{K})}\\
        &\leq\chi(\mathbf{K})^{m-1}.
    \end{align*} This leads to a contradiction, as $\chi(\mathbf{K})^{m-1}+1\leq \chi(\mathbf{L})$.
\end{proof}

%%%%%%%%%%%%%%%%%%%%%%%%%%%%%%%%%%%%%%%%%%%%%%%%%%%%%%%%%%%%%%%%%%%%%%%%%%%%%%%%%%%%%%%%%%

\subsection{Upper bound} Let $\mathsf{L}$ be a simplicial complex. Let \[\eta(\mathsf{L})=|\{F\in\mathsf{L}:~\text{ $F$ is a facet of $\mathsf{L}$}\}|,\] the number of facets of $\mathsf{L}$. We have the following statement.

\begin{proposition}\label{prop:upper-bound-facets}
 Let $\mathsf{L}$ be a simplicial complex. For a fixed complete simplicial complex $\mathsf{K}=2^V$ with $|V|\geq 2$, if $|F|\geq |V|$ for any facet $F$ in $\mathsf{L}$, then \[\mathrm{C}(\mathsf{L};\mathsf{K})\leq \eta(\mathsf{L}).\]  
\end{proposition}
\begin{proof}
    For each facet $F$ of $\mathsf{L}$, let us denote the underlying subcomplex of $F$ by $\mathsf{L}_F$, that is, \[\mathsf{L}_F=\{S\in \mathsf{L}:~S\subseteq F\}.\] Notice that $F$ is the only facet of $\mathsf{L}_F$, $V(\mathsf{L}_F)=F$, and $\mathsf{L}=\bigcup_{\eta(\mathsf{L})}\mathsf{L}_F$. 

    For any facet $F$ in $\mathsf{L}$, if $|F|\geq |V|$ there exists an injection $j:V\to F$. One can construct a facet simplicial map as follows. 
    
    First of all, let $v_0\in V$. Then, define $f: \mathsf{L}_F\to \mathsf{K}$, for $x\in F$, \[f(x)=\begin{cases}
        v,&\hbox{ if $x=j(v)$ for some $v\in V$;}\\
        v_0,&\hbox{ if $x\neq j(v)$ for all $v\in V$}.
    \end{cases}\] Notice that the $v$ in the first case is uniquely determined because $j$ is injective. Hence, $f$ is a simplicial map (here we use the fact that $\mathsf{K}$ is complete) and $f(F)=V$. Therefore, $f$ is a facet simplicial map. So, it follows that $\mathrm{C}(\mathsf{L};\mathsf{K})\leq \eta(\mathsf{L})$.
\end{proof}

 Note that if $|V|=2$, then the condition that $|F|\geq |V|$ for any facet $F$ in $\mathsf{L}$ corresponds to the condition that $\mathsf{L}$ has no isolated vertices. Hence, we have the following corollary.

\begin{corollary} If $\mathsf{L}$ has no isolated vertices, then \[\mathrm{C}(\mathsf{L};\Gamma_2)\leq \eta(\mathsf{L}).\]
\end{corollary}

Recall that a subcomplex $\mathsf{K}$ of $\mathsf{H}$ is said to be a facet subcomplex if any facet of $\mathsf{K}$ is a facet of $\mathsf{H}$. We have the following result.

\begin{theorem}[Upper Bound]\label{thm:uper-bound} For any simplicial complexes $\mathsf{L}$ and $\mathsf{H}$. Let $G_0$ be a facet of $\mathsf{H}$ such that $2\leq |G_0|\leq |G|$ for any facet $G$ in $\mathsf{H}$. One of the following holds.
\begin{enumerate}
\item[(1)] $\mathrm{C}(\mathsf{L};\mathsf{H})\leq \eta(\mathsf{L})$, if $|F|\geq |G_0|$ for any facet $F$ in $\mathsf{L}$.
\item[(2)] $\mathrm{C}(\mathsf{L};\mathsf{H})=\infty$, if otherwise. 
\end{enumerate}
\end{theorem}
\begin{proof}
    \noindent\begin{enumerate}
        \item[(1)]  For the facet $G_0$ of $\mathsf{H}$, one can define the facet (complete) subcomplex $\mathsf{H}_{G_0}=2^{G_0}$ of $\mathsf{H}$ as in Proposition~\ref{prop:upper-bound-facets}. By Theorem~\ref{prop:complexity-subgraphs}, we have $\mathrm{C}(\mathbf{L};\mathbf{H})\leq \mathrm{C}(\mathbf{L};\mathsf{H}_{G_0})$. Furthermore, by Proposition~\ref{prop:upper-bound-facets}, $\mathrm{C}(\mathbf{L};\mathsf{H}_{G_0})\leq\eta(\mathsf{L})$. Hence, we obtain $\mathrm{C}(\mathsf{L};\mathsf{H})\leq \eta(\mathsf{L})$.
        \item[(2)] Without loss of generality, suppose that there is only one facet in $\mathsf{L}$ whose cardinality is strictly less than $|G_0|$ and name it $F_1$. 

Suppose that $\mathrm{C}(\mathsf{L};\mathsf{H})=k<\infty$. That is, there are subcomplexes $\mathsf{L}_1,\ldots, \mathsf{L}_k$ of $\mathsf{L}$ such that $\mathsf{L}=\bigcup_{i=1}^{k} \mathsf{L}_i$ and for all $i=1,\ldots, k$, $f_i:\mathsf{L}_i\to \mathsf{H}$ is a facet simplicial map.

Since $\mathsf{L}_i$'s cover $\mathsf{L}$, then at least one $\mathsf{L}_i$ must contain $F_1$. Without loss of generality, say that $\mathsf{L}_1$ is the only subcomplex containing $F_1$. We assumed that there is a facet simplicial map $f_1: \mathsf{L}_1\to \mathsf{H}$. Then, all the facets of $L_1$ must be mapped to a facet of $\mathsf{H}$. But $F_1$ cannot be mapped to a facet of $\mathsf{H}$ as $|F_1| < |G|$ for any facet $G$ of $H$. This completes the proof. 
    \end{enumerate}
\end{proof}

A simplicial complex $\mathsf{L}$ is said to be \textit{pure} if $\dim \mathsf{L}<\infty$ and every facet is of dimension $\dim \mathsf{L}$. For the injective complexity, we have the following result.

\begin{proposition}\label{prop:k-complete-facet} Let $\mathsf{L}$ be a simplicial complex such that $\eta(\mathsf{L})<\infty$. For a fixed complete simplicial complex $\mathsf{K}=2^V$ with $|V|\geq 2$, we have \[\eta(\mathsf{L})\leq \mathrm{IC}(\mathsf{L};\mathsf{K}).\] Furthermore, the equality holds whenever $\mathsf{L}$ is a pure simplicial complex with $\dim \mathsf{L}=\dim \mathsf{K}$.
\end{proposition}

\begin{proof}
If $\eta(\mathsf{L})=1$, the inequality $\eta(\mathsf{L})\leq \mathrm{IC}(\mathsf{L};\mathsf{K})$ always holds. Hence, we assume $\eta(\mathsf{L})\geq 2$. 

We will check that $\eta(\mathsf{L})\leq \mathrm{IC}(\mathsf{L};\mathsf{K})$. By contradiction, assume $\mathrm{IC}(\mathsf{L};\mathsf{K})\leq \eta(\mathsf{L})-1:=k$. Then, there exist subcomplexes $\mathsf{L}_1,\ldots, \mathsf{L}_k$ of $\mathsf{L}$ such that $\mathsf{L}=\bigcup_{i=1}^{k} \mathsf{L}_i$ and for all $i=1,\ldots, k$, $f_i:\mathsf{L}_i\to \mathsf{K}$ is an injective facet simplicial map. Since $\eta(\mathsf{L})=k+1$ there exist two different facets $F$ and $G$ of $\mathsf{L}$ such that $F,G\in \mathsf{L}_j$ for some $j\in\{1,\ldots,k\}$. Notice that $F$ and $G$ are facets of $\mathsf{L}_j$, then $f_j(F)=V=f_j(G)$, and thus $F=G$, which is a contradiction. Therefore, we conclude $\eta(\mathsf{L})\leq \mathrm{IC}(\mathsf{L};\mathsf{K})$.   

On the other hand, let $\eta(\mathsf{L})=m$. For each facet $F$ of $\mathsf{L}$, one can define the facet subcomplex $\mathsf{L}_F$ as in Proposition~\ref{prop:upper-bound-facets}. If $\mathsf{L}$ is a pure simplicial complex with $\dim \mathsf{L}=\dim \mathsf{K}$, then $|F| = |V|<\infty$ for each facet $F$ of $\mathsf{L}$. So, similarlly as in Proposition~\ref{prop:upper-bound-facets}, there is an injective facet simplicial map $f:\mathsf{L}_F\to \mathsf{K}$. Thus, we obtain that $\mathrm{IC}(\mathsf{L};\mathsf{K})\leq \eta(\mathsf{L})$ which completes the proof.

\end{proof}

Finally, we propose the following future work.

\begin{remark}[Future Work]\label{rem:future-work}
 \noindent\begin{enumerate}
     \item[(1)] Based on Remark~\ref{rem:ineq-complex-q}, we propose the following question: Do the equalities \[\text{C}_s(\mathsf{L};\mathsf{K})=\cdots= \text{C}_s(\mathsf{L}^{(2)};\mathsf{K}^{(2)})=\text{C}(\mathsf{L}^{\ast};\mathsf{K}^{\ast})\] and \[\text{IC}_s(\mathsf{L};\mathsf{K})=\cdots= \text{IC}_s(\mathsf{L}^{(2)};\mathsf{K}^{(2)})=\text{IC}(\mathsf{L}^{\ast};\mathsf{K}^{\ast})\] always hold for any simplicial complexes  $\mathsf{L}$ and $\mathsf{K}$?
     \item[(2)] Let $\mathsf{L}$ and $\mathsf{K}$ be simplicial complexes. Given a facet simplicial map $f:\mathsf{K}\to \mathsf{L}$, we have \[\mathrm{C}(\mathsf{L};\mathsf{K})\leq \mathrm{IC}(\mathsf{L};\mathsf{K})\leq \text{sec}(f).\] Here, $\text{sec}(f)$ denotes the sectional number of $f$ which is a higher version of the sectional number of a group homomorphism introduced in \cite{zapata2023}. Specifically, $\text{sec}(f)$ is the least positive integer $k$ such that there exist facet subcomplexes $\mathsf{L}_1\ldots,\mathsf{L}_k$ of $\mathsf{L}$ with $\mathsf{L}=\mathsf{L}_1\cup\cdots\cup \mathsf{L}_k$, and for each $\mathsf{L}_i$, there exists a facet simplicial map $\sigma_i:\mathsf{L}_i\to \mathsf{K}$ such that $f\circ\sigma_i=\mathrm{incl}_{\mathsf{L}_i}$ (and thus each $\sigma_i:\mathsf{L}_i\to \mathsf{K}$ is an injective facet simplicial map), where $\mathrm{incl}_{\mathsf{L}_i}:\mathsf{L}_i\hookrightarrow \mathsf{L}$ is the inclusion facet simplicial map. We propose studying this notion of sectional number further.  
\end{enumerate}   
\end{remark}

\section*{Acknowledgment}
The first author would like to thank grants \#2022/16695-7 and \#2023/16525-7 from the S\~{a}o Paulo Research Foundation (\textsc{fapesp}) for financial support.

\section*{Conflict of Interest Statement}
The authors declare that there are no conflicts of interest.

\bibliographystyle{plain}

\begin{thebibliography}{11}
\bibitem{golubev2017} Golubev, K. (2017). On the chromatic number of a simplicial complex. Combinatorica, 37(5), 953-964.
\bibitem{harary1970} Harary, F. (1970). Covering and packing in graphs, I. Annals of the New York Academy of Sciences, 175(1), 198-205.
\bibitem{hell1990} Hell, P., \& Nešetřil, J. (1990). On the complexity of H-coloring. Journal of Combinatorial Theory, Series B, 48(1), 92-110.
\bibitem{hussein2021} Hussein, A. A. (2021). Data migration: Need, strategy, challenges, methodology, categories, risks, uses with cloud computing, and improvements using suggested proposed Model (DMig1). Journal of Information Security, 12, 79-103.
\bibitem{matousek2003} Matoušek, J. (2003). Using the Borsuk-Ulam theorem: lectures on topological methods in combinatorics and geometry. Springer.
\bibitem{spivak2012} Spivak, D. I. (2012). Functorial data migration. Information and Computation, 217, 31-51. 
\bibitem{zapata2024} Zapata, C. A. I., Enciso, J. A. A., \& Ramos, W. F. C. (2024). (Injective) facet-complexity between graphs. arXiv preprint arXiv:2411.16547v2.
\bibitem{zapata2023} Zapata, C. A. I., \& Ramos, W. F. C. (2023). Número seccional de un homomorfismo de grafos. Pesquimat, 26(2), 39-46.
\end{thebibliography}

%%%%%%%%%%%%%%%%%%%%%%%%%%%%%%%%%%%%%%%%%%%
\end{document}